\documentclass{amsart} %poner draft para ver los too-wides
\usepackage[utf8]{inputenc}
\usepackage{slashed}
\usepackage{textcomp} % provide lots of new symbols
\usepackage{amsmath,amstext,amsopn,amssymb, amsfonts, amsthm, mathtools}  % Better maths support & more symbols
\usepackage{mathrsfs}
\usepackage{graphicx}
\usepackage{xcolor}
\usepackage{geometry} % see geometry.pdf on how to lay out the page. There's lots.
\usepackage{array} % for better arrays (eg matrices) in maths
\usepackage{textcomp} % provide lots of new symbols
\usepackage[all,cmtip]{xy} % conmutative diagrams
\usepackage{bbm}
\usepackage{varioref}
\usepackage{pdfsync}
\usepackage{enumitem}
\makeatletter
\def\namedlabel#1#2{\begingroup
    #2%
    \def\@currentlabel{#2}%
    \phantomsection\label{#1}\endgroup
}
\makeatother

\usepackage[pdftex,bookmarks,colorlinks,breaklinks]{hyperref}  % PDF hyperlinks, with coloured links
\definecolor{dullmagenta}{rgb}{0.4,0,0.4}   % #660066
\definecolor{darkblue}{rgb}{0,0,0.4}
\definecolor{darkgreen}{rgb}{0,0.4,0}
\hypersetup{linkcolor=darkblue,citecolor=blue,filecolor=dullmagenta,urlcolor=darkblue} % coloured links
\newcommand{\essinf}{\operatorname{ess \, inf}}

\newcommand{\supp}{ \operatorname{supp}}

\newcommand{\car}{\operatorname{Car}}

\newcommand{\overlineguay}[2]{\overline{#1\phantom{\setbox10=\hbox{\tiny$#2$} \box10}}\ensuremath{\setbox0=\hbox{$\overline{\mathbb{#1}}$}
\raise\ht0\vbox{ \setbox11=\hbox{\tiny$#2$} \moveleft\the\wd11\hbox{\raise-1.25ex\hbox{\tiny$#2$}}}}}

\def\XXint#1#2#3{{\setbox0=\hbox{$#1{#2#3}{\int}$}
     \vcenter{\hbox{$#2#3$}}\kern-.5\wd0}}

\newtheorem{theorem}{Theorem}[section]
\newtheorem*{theorem*}{Theorem}
\newtheorem{lemma}[theorem]{Lemma}
\newtheorem*{lemma*}{Lemma}
\newtheorem{proposition}[theorem]{Proposition}

\theoremstyle{definition}
\newtheorem{definition}[theorem]{Definition}

\theoremstyle{remark}
\newtheorem{remark}[theorem]{Remark}
\newtheorem{question*}[theorem]{Question}

\numberwithin{equation}{section}
\theoremstyle{theorem}
\newtheorem{ltheorem}{Theorem}

\newtheorem{ltheoremprime}{Theorem}

\newtheorem{lcorollaryprime}[ltheoremprime]{Corollary}

\title{A pointwise estimate for positive dyadic shifts and some applications}

\author{Jos\'e M. Conde-Alonso}
\address{Instituto de Ciencias Matem\'aticas, CSIC-UAM-UC3M-UCM, C/ Nicol\'as Cabrera 13-15. 28049, Madrid. Spain}
\email{jose.conde@icmat.es}
\thanks{J.M. Conde-Alonso was partially supported by the ERC StG-256997-CZOSQP, the Spanish grant MTM2010-16518 and by ICMAT Severo Ochoa Grant SEV-2011-0087 (Spain).}

\author{Guillermo Rey}
\address{Department of Mathematics, Michigan State University, East Lansing MI 48824-1027}
\email{reyguill@math.msu.edu}

\makeindex

\begin{document}

\maketitle

\begin{abstract}
    We prove a pointwise estimate for positive dyadic shifts of complexity $m$ which is linear in the complexity. This can be used to give a pointwise estimate for
    Calder\'on-Zygmund operators and to answer a question posed by A. Lerner. Several applications to weighted estimates for both multilinear Calder\'on-Zygmund operators
    and square functions are discussed.
\end{abstract}

\section{Introduction}\label{Complexity.Introduction}

One particularly useful way to study singular integrals (or maximal operators) is that of decomposing them into sums of simpler dyadic operators. One example of a
recent striking result using this strategy
is the proof of the sharp weighted estimate for the Hilbert transform by S. Petermichl \cite{Petermichl2007}, which was a key step towards 
the full $A_2$ theorem for general Calder\'on-Zygmund operators,
finally proven by T. Hyt\"onen in \cite{Hytonen2012a}. Of course there are many instances of this useful technique, but we will not try to give a thorough historical overview here.

The proof in \cite{Hytonen2012a} was a \textit{tour de force} which was the culmination of many previous partial efforts by others, see \cite{Hytonen2012a} and the references therein. Hyt\"onen did not only prove the $A_2$ theorem, but he
also showed that general Calder\'on-Zygmund operators could be represented as averages of certain simpler ``Haar shifts'' in the spirit of \cite{Petermichl2007}. The sharp weighted bound then followed from the corresponding one for these simpler operators.

Later, A. Lerner gave a simplification of the $A_2$ theorem in \cite{Lerner2013a} which avoided the use of most of the complicated machinery in \cite{Hytonen2012a}; it mainly relied on a general pointwise
estimate for functions in terms of positive dyadic operators which had already been proven in \cite{Lerner2010}. The weighted result for the positive dyadic shifts that this contribution reduced the problem to had been already shown before in \cite{Lacey2009}, see also
\cite{Cruz-Uribe2010} and \cite{Cruz-Uribe2012}. More precisely, the proof of Lerner (essentially) gave the following pointwise estimate for general Calder\'on-Zygmund operators $T$: for every dyadic cube $Q$
\begin{equation} \label{Complexity.PointwiseHPL}
    |Tf(x)| \lesssim \sum_{m=0}^\infty 2^{-\delta m} \mathcal{A}^m_{\mathcal{S}}|f|(x) \quad \text{for a.e. } x\in Q,
\end{equation}
where $\delta > 0$ depends on the operator $T$, $\mathcal{S}$ are collections of dyadic cubes (belonging to same dyadic grid for each fixed $\mathcal{S}$) which depend on $f$, $T$ and $m$, 
and $\mathcal{A}^m_{\mathcal{S}}$ are positive dyadic operators defined by
\begin{equation*}
    \mathcal{A}_{S}^m f(x) = \sum_{Q \in \mathcal{S}} \langle f \rangle_{Q^{(m)}} \mathbbm{1}_Q(x),
\end{equation*}
where $Q^{(m)}$ is the $m$-th dyadic parent of $Q$. The collections $\mathcal{S}$ in \eqref{Complexity.PointwiseHPL} are \textit{sparse} in the following sense: given $0 < \eta< 1$, we say that a collection of cubes $\mathcal{S}$ belonging to the same daydic grid is $\eta$-sparse if for all cubes $Q \in \mathcal{S}$
    there exists measurable subsets $E(Q) \subset Q$ with $|E(Q)| \geq \eta|Q|$ and $E(Q) \cap E(Q') = \emptyset$ unless $Q = Q'$. A collection is called simply sparse if it is
    $\frac{1}{2}$-sparse.

From this pointwise estimate Lerner continues the proof by showing that bounding the operator norm of each $\mathcal{A}^m_\mathcal{S}$ can be reduced to just estimating the operator norm of $\mathcal{A}^0_{S'}$ in the same space
for all possible sparse collections $\mathcal{S}'$. More precisely, he shows that
\begin{equation} \label{Complexity.LinearInM}
    \|\mathcal{A}^m_{\mathcal{S}}f\|_X \lesssim (m+1) \sup_{\mathscr{D}, \mathcal{S'}} \|\mathcal{A}^0_{\mathcal{S}'}f\|_X,
\end{equation}
where the supremum is taken over all dyadic grids $\mathscr{D}$ and all sparse collection $\mathcal{S}' \subset \mathscr{D}$, and where $X$ is any Banach function space (see Chapter 1 of \cite{Bennett1988}).

It is at this point where the duality of $X$ is needed in the argument; the operators $\mathcal{A}^m_\mathcal{S}$ do not lend themselves to Lerner's pointwise formula, while their adjoints do. Consequently, the
question of what to do when no duality is present was left open. Our main result answers this question by proving a stronger (though localized) statement: the operators $\mathcal{A}^m_\mathcal{S}$
are actually pointwise bounded by positive dyadic $0$-shifts:
\begin{ltheorem} \label{Complexity.TheoremA}
    Let $P$ be a dyadic cube and $\mathcal{S}$ a sparse collection of dyadic cubes $Q$  such that $Q^{(m)} \subseteq P$,
    then for all nonnegative integrable functions $f$ on $P$ there exists
    another sparse collection $\mathcal{S}'$ of dyadic subcubes of $P$ such that
    \begin{equation} \label{Complexity.TheoremA.eq}
        \mathcal{A}^m_{\mathcal{S}} f(x) \lesssim (m+1) \mathcal{A}^0_{\mathcal{S'}}f(x) \quad \forall x \in P
    \end{equation}
\end{ltheorem}
In fact, what we will do is to prove Theorem \ref{Complexity.TheoremA} in a slightly more general setting: first, the statement is proven for a certain natural multilinear generalization of the operators $\mathcal{A}^m_\mathcal{S}$. Also, the sparse
collection $\mathcal{S}$ is replaced by a more general Carleson sequence (see the next section for details).

The novelty in our approach is two-fold: we directly attack the pointwise estimate for the operators $\mathcal{A}^m$, instead of bounding their norm in various spaces.
Also, in proving the pointwise bound we develop an algorithm that constructively selects those cubes which will form the family $\mathcal{S}'$. This algorithm has
some ``memory'' and each iteration takes into account the previous steps, a feature which is crucial in our method to ensure that $\mathcal{S}$ is sparse.

As a corollary of Theorem \ref{Complexity.TheoremA}, we find an analogue of \eqref{Complexity.PointwiseHPL} for Calder\'on-Zygmund operators with more general moduli of continuity (see the next section
for the precise definition). In particular, we obtain the following
pointwise estimate for Calder\'on-Zygmund operators:
\begin{lcorollaryprime} \label{Complexity.TheoremA.1}
    If $P$ is a dyadic cube,
    $f$ is an integrable function supported on $P$ and $T$ is a Calder\'on-Zygmund
    operator with modulus of continuity $\omega$, then
    \begin{equation}
        |Tf(x)| \lesssim \sum_{m=0}^\infty \omega(2^{-m}) (m+1) \mathcal{A}^0_{\mathcal{S}_m} |f|(x) \quad \text{for a.e. }x\in P,
    \end{equation}
    where $\mathcal{S}_m$ are sparse collections belonging to at most $3^d$ different dyadic grids.

    Moreover, if we know that $\omega$ satisfies the logarithmic Dini condition:
    \begin{equation}
        \label{eq:Complexity.logDini}
        \int_0^1 \omega(t) \Bigl( 1 + \log\Bigl(  \frac{1}{t} \Bigr)  \Bigr) \frac{dt}{t} < \infty,
    \end{equation}
    then we can find sparse collections $\{\mathcal{S}_1', \dots, \mathcal{S}_{3^d}'\}$, belonging to possibly different dyadic grids, such that
    \begin{equation}
        |Tf(x)| \lesssim \sum_{i=1}^{3^d} \mathcal{A}^0_{\mathcal{S}_i'} |f|(x) \quad \text{for a.e. } x\in P.
    \end{equation}
\end{lcorollaryprime}

The factor $m$ in \eqref{Complexity.LinearInM} precluded a naive adaptation of the proof in \cite{Lerner2013} to an $A_2$ theorem with the usual Dini condition:
\begin{equation}
    \label{Complexity.Dini}
    \int_0^1 \omega(t) \frac{dt}{t} < \infty.
\end{equation}

since the sum
\begin{equation} \label{Complexity.LogDini}
    \sum_{m=0}^\infty \omega(2^{-m})(m+1) \simeq \int_0^1 \omega(t) \Bigl(1+\log\frac{1}{t} \Bigr) \frac{dt}{t}
\end{equation}
could diverge for some moduli $\omega$ satisfying \eqref{Complexity.Dini}. Moreover, it was shown in \cite{Hytonen2012c} that the weak-type $(1,1)$ norm of
the adjoints of the operators $\mathcal{A}^m_\mathcal{S}$ was at least linear in $m$, even in the unweighted case, so using duality prevented an extension of this type. However, although our argument does not quite give an $A_2$ theorem for Calder\'on-Zygmund operators satisfying the Dini condition (we still need \eqref{Complexity.LogDini} to be finite),
our proof avoids the use of duality and the study of the adjoint operators $(\mathcal{A}_{\mathcal{S}}^m)^*$. It thus removes at least one of the obstructions to possible proofs
of the $A_2$ theorem with the Dini condition which follow this strategy. Hence, removing the linear factor of $m$ in Theorem \ref{Complexity.TheoremA} remains as an interesting open problem.

Apart from being interesting in its own right, a bound for Calder\'on-Zygmund operators by these sums 
of positive $0$-shifts in cases where there is no duality has interesting applications, some of which we describe later. First, let us state a second corollary to Theorem \ref{Complexity.TheoremA.1}:

\begin{lcorollaryprime} \label{Complexity.TheoremA.2}
    Let $\| \cdot \|_X$ be a function quasi-norm (see section \ref{Complexity.Pointwise}) and $T$ a Calder\'on-Zygmund operator satisfying the logarithmic Dini condition, then
    \begin{equation}
        \|Tf\|_X \lesssim \sup_{\mathscr{D}, \mathcal{S}} \|\mathcal{A}^0_\mathcal{S}|f|\|_X,
    \end{equation}
    where the supremum is taken over all dyadic grids $\mathscr{D}$ and all sparse collections $\mathcal{S} \subset \mathscr{D}$.
\end{lcorollaryprime}

We now describe two immediate applications of our result. First we can continue the program, initiated in \cite{Damian2012} and extended in \cite{Li2014}, which aims to extend the sharp
weighted estimates for Calder\'on-Zygmund operators to their multilinear analogues (as in \cite{Grafakos2002}). In particular we obtain
\begin{ltheorem} \label{Complexity.TheoremB}
    Let $T$ be a multilinear Calder\'on-Zygmund operator (see section \ref{Complexity.Pointwise} for the relevant definitions).
    Suppose $1 < p_1, \dots, p_k < \infty$, $\frac{1}{p} = \frac{1}{p_1} + \dots + \frac{1}{p_k}$ and $\vec{w} \in A_{\vec{P}}$. Then
    \begin{equation}
        \|T\vec{f}\|_{L^p(v_{\vec{w}})} \lesssim [\vec{w}]_{A_{\vec{P}}}^{\max(1, \frac{p_1'}{p}, \dots, \frac{p_k'}{p})}\prod_{i=1}^k \|f_i\|_{L^p(w_i)}.
    \end{equation}
\end{ltheorem}
The same theorem was proven in \cite{Li2014} but with the additional hypothesis that $p$ had to be at least $1$. The proof of this Theorem is an application of the result in
\cite{Li2014} which proved the same result (without the condtion $p \geq 1$) but for a multilinear analogue of the operators $\mathcal{A}^m_{\mathcal{S}}$, together with 
Theorem \ref{Complexity.TheoremA}. In fact, we will need a multilinear version of Theorem \ref{Complexity.TheoremA} which we state and prove in the next
section.

The second application is a sharp aperture weighted estimate for square functions which extends a result in \cite{Lerner2014}. In particular (see section \ref{Complexity.Applications}
for the relevant definitions), we have:
\begin{ltheorem} \label{Complexity.TheoremC}
    Let $\alpha > 0$, then the square function $S_{\alpha,\psi}$ for the cone in $\mathbbm{R}^{d+1}_+$  of apperture $\alpha$ and the standard kernel $\psi$ satisfies
    \begin{equation*}
        \|S_{\alpha,\psi}f\|_{L^{p,\infty}(\mathbbm{R}^d,w)} \lesssim \alpha^d [w]_{A_p}^{1/p} \|f\|_{L^p(\mathbbm{R}^d,w)} \quad \text{for }1 < p < 2
    \end{equation*}
    and
    \begin{equation} \label{Complexity.TheoremC.EndpointWT}
        \|S_{\alpha,\psi}f\|_{L^{2,\infty}(\mathbbm{R}^d,w)} \lesssim \alpha^d [w]_{A_2}^{1/2} (1+\log[w]_{A_2}) \|f\|_{L^2(\mathbbm{R}^d,w)}.
    \end{equation}
\end{ltheorem}
An analogous result was shown in \cite{Lerner2014} for $2 < p< 3$:
\begin{equation*}
    \|S_{\alpha,\psi}f\|_{L^{p,\infty}(\mathbbm{R}^d,w)} \lesssim \alpha^d [w]_{A_p}^{1/2} (1+\log[w]_{A_2}) \|f\|_{L^p(\mathbbm{R}^d,w)}.
\end{equation*}
The proof relies on the use of Lerner's pointwise formula and previous results by Lacey and Scurry \cite{Lacey2012}. However, in \cite{Lerner2014} the requirement of $p>2$ was necessary for the same
reason the proof in the multilinear weighted estimates required $p \geq 1$ (a certain space had no satisfactory duality properties). Theorem \ref{Complexity.TheoremA} can be used
in almost the same way as with the weighted multilinear estimates to prove Theorem \ref{Complexity.TheoremC}. Indeed, the proofs in \cite{Lacey2012} and \cite{Lerner2014} reduce the problem
to estimating certain discrete positive operators which can be seen to be particular instances of the positive multilinear $m$-shifts used in the proof of Theorem \ref{Complexity.TheoremB}.

As was noted in \cite{Lacey2012}, estimate \eqref{Complexity.TheoremC.EndpointWT} can be seen as an analogue of the result in \cite{Lerner2009} stablishing the endpoint
weighted weak-type estimate for Calder\'on-Zygmund operators (see also \cite{NRVV-WeightedHaarShifts} for a similar estimate from below and more information on the sharpness of this estimate, known as the
weak $A_1$ conjecture). Indeed, the proof in \cite{Lacey2012} can be adapted to the operators $\mathcal{A}^m_{\mathcal{S}}$ and combined with our pointwise bound to give a simpler
proof of the results in \cite{Lacey2012} which in particular avoids the use of extrapolation, as well as other techniques. In this direction, 
it seems reasonable that Lacey and Scurry's proof in \cite{Lacey2012} could be adapted to the multilinear setting, however we will not pursue this problem here.

 Finally, as a third application of our results, it is possible to to give a more direct proof of the result in \cite{Hytonen2013} for the $q$-variation of Calder\'on-Zygmund operators satisfying the logarithmic Dini condition by using the pointwise estimate analogous to \eqref{Complexity.PointwiseHPL} in \cite{Hytonen2013} and then applying Theorem \ref{Complexity.TheoremA}. We will, however, not pursue this argumentation either.

Shortly before uploading this preprint, Andrei Lerner kindly communicated to the authors that he, jointly with Fedor Nazarov, had independently proven 
a theorem very similar to Corollary \ref{Complexity.TheoremA.1} \cite{LernerNazarov}. Though the hypothesis are the same, their result differs from the one in this note in that we give a localized pointwise estimate while 
their pointwise estimate is valid for all of $\mathbbm{R}^d$. However, our result seems to be as powerful in the applications. 
\section*{Acknowledgements}

The authors wish to thank Javier Parcet, Ignacio Uriarte-Tuero and Alexander Volberg for insightful discussions, and Andrei Lerner and Fedor Nazarov for sharing with us the details of
their construction.

\section{Pointwise domination} \label{Complexity.Pointwise}

The goal of this section is the proof of Theorem A and its consequences as stated in the introduction.
We will prove the result in the level of generality of multilinear operators. Given a cube $P_0$ on $\mathbb{R}^d$, we will denote by
$\mathscr{D}(P_0)$ the dyadic lattice obtained by successive dyadic subdivisions of $P_0$.
By a dyadic grid we will denote any dyadic lattice composed of cubes with sides parallel to the axis.
A $k$-linear positive dyadic shift of complexity $m$ is an operator of the form
\[
		\mathcal{A}_{P_0, \alpha}^m \vec{f}(x) = \mathcal{A}_{P_0, \alpha}^m (f_1,f_2, \cdots, f_k)(x) := \sum_{Q \in \mathscr{D}(P_0)} \alpha_Q \Bigl( \prod_{i=1}^k \langle f_i \rangle_{Q^{(m)}} \Bigr) \mathbbm{1}_Q(x).
	\]	
	
	As a first step towards the proof of Theorem A, it is convenient to separate the scales of (or \textit{slice}) $\mathcal{A}^m_{\alpha,P_0}$ as follows:
\begin{align*}
	\mathcal{A}^m_{\alpha,P_0}\vec{f}(x) &= 
		\sum_{n=0}^{m-1} \sum_{j=1}^\infty \sum_{Q \in \mathscr{D}_{jm+n}(P_0)} \alpha_Q \Bigl( \prod_{i=1}^k \langle f_i \rangle_{Q^{(m)}} \Bigr) \mathbbm{1}_Q(x) \\
	&=: \sum_{n=0}^{m-1} \mathcal{A}_{P_0,\alpha}^{m,n} \vec{f}(x).
\end{align*}

Note that $\mathscr{D}_k(P_0)$ denotes the $k$-th generation of the lattice $\mathscr{D}(P_0)$. Now we rewrite $\mathcal{A}_{P_0,\alpha}^{m;n}$ 
as a sum of disjointly supported operators of the form $\mathcal{A}_{P,\alpha}^{m;0}$. Indeed,
\begin{align*}
	\mathcal{A}_{P_0,\alpha}^{m;n} \vec{f}(x) & = \sum_{j=1}^\infty \sum_{Q \in \mathscr{D}_{jm+n}(P_0)} \alpha_Q \Bigl( \prod_{i=1}^k \langle f_i \rangle_{Q^{(m)}} \Bigr) \mathbbm{1}_Q(x) \\ & = 
		\sum_{P \in \mathscr{D}_n(P_0)} \sum_{j=1}^\infty \sum_{Q \in \mathscr{D}_{jm}(P)} \alpha_Q \Bigl( \prod_{i=1}^k \langle f_i \rangle_{Q^{(m)}} \Bigr) \mathbbm{1}_Q(x) \\
	&= \sum_{P \in \mathscr{D}_n(P_0)} \mathcal{A}_{P,\alpha}^{m;0} \vec{f}(x),
\end{align*}
	
	which leads to the expression
	\begin{align*}
		\mathcal{A}^m_{\alpha,P_0}\vec{f}(x) &= \sum_{n=0}^{m-1} \sum_{P \in \mathscr{D}_n(P_0)} \mathcal{A}^{m;0}_{P,\alpha} \vec{f}(x).
	\end{align*}

        We say that a sequence $\{\alpha_Q\}_{Q \in \mathscr{D}(P_0)}$ is Carleson if its Carleson constant $\|\alpha\|_{\text{Car}(P_0)} < \infty$, where
        \begin{equation*}
           \|\alpha\|_{\text{Car}(P_0)} = \sup_{P \in \mathscr{D}(P_0)} \frac{1}{|P|} \sum_{Q \in \mathscr{D}(P)} \alpha_Q|Q|.
        \end{equation*}

	The following intermediate step is the key to our approach:
\begin{proposition} \label{Complexity:Pointwise.MainTheorem}
	Let $m\geq 1$ and $\alpha$ be a Carleson sequence. For integrable functions $f_1, \dots, f_k \geq 0$ on $P_0$ 
	there exists a sparse collection $\mathcal{S}$ of cubes in $\mathscr{D}(P_0)$ such that
	\[
		\mathcal{A}^{m;0}_{P_0,\alpha} \vec{f}(x) \leq C_1\|\alpha\|_{\car(P_0)} \sum_{Q \in \mathcal{S}} \Bigl( \prod_{i=1}^k\langle f_i \rangle_Q \Bigr) \mathbbm{1}_Q(x),
	\]
	where $C_1$ only depends on $k$ and $d$, and in particular is independent of $m$.
\end{proposition}

To prove Proposition \ref{Complexity:Pointwise.MainTheorem} we will proceed in three steps:
we will first construct the collection $\mathcal{S}$, then show that we have the required pointwise bound,
and finally that $\mathcal{S}$ is sparse. By homogeneity, we will assume that $\|\alpha\|_{\car(P_0)} = 1$.
Also, we will assume that the sequence $\alpha$ is finite, but our constants will be independent of the number of elements in the sequence.

Let $\Delta_{P_0} = 0$ and, for each $Q \in \mathscr{D}_{mj}(P_0)$ with $j \geq 0$, define the sequence $\{\gamma_Q\}_Q$ by
\[
	\gamma_Q = \max_{R \in \mathscr{D}_m(Q)} \alpha_R.
\]

For each $Q \in \mathscr{D}_{mj}(P_0)$ with $j \geq 0$, we will inductively define the quantities $\Delta_Q$ and $\beta_Q$ as follows:
\[
	\beta_Q =
		\begin{cases}
			0 &\text{if } \Delta_Q - \Bigl( \prod_{i=1}^k \langle f_i \rangle_Q \Bigr) \gamma_Q \geq 0 \\
			2^{2(k+1)}C_W &\text{otherwise},
		\end{cases}
\]
where $C_W$ is the boundedness constant of the unweighted endpoint weak-tpe of the operators $\mathcal{A}^m$ proved in Theorem \ref{Complexity.LebesgueWeakType.Theorem} in the appendix. Also, for every $R \in \mathscr{D}_m(Q)$ we define
\[
	\Delta_R = \Delta_Q + (\beta_Q - \alpha_R)\Bigl( \prod_{i=1}^k \langle f_i \rangle_Q \Bigr).
\]
Note that the definition only applies to cubes in $\mathscr{D}_{mj}(P_0)$ for some $j$. For all other cubes in $\mathscr{D}_{P_0}$, we set $\beta_Q= \Delta_Q = 0$. The collection $\mathcal{S}$ consists of those cubes $Q \in \mathscr{D}(P_0)$ for which $\beta_Q \neq 0$. Note that, since $2^{2(k+1)}C_1 > 1 = \|\alpha\|_{\car(P_0)} \geq \alpha_R$ for all $R$ and by the definition of $\gamma_Q$, we must have $\Delta_Q \geq 0$ for all $Q$. This can be easily seen by induction. 

\begin{remark}
Observe that the quantity $\Delta_Q$ is what ultimately allows us to construct the correct dominating operator.
It can be seen as a kind of ``memory'' of the algorithm, which allows the construction of the sequence $\{\beta_Q\}_Q$ to be more judicious.
\end{remark}

\begin{lemma}
\label{lemmaunodos}
	We have the pointwise bound
	\begin{equation} \label{Complexity:Pointwise.Lemma.eq}
		\mathcal{A}^{m;0}_{P_0,\alpha} \vec{f}(x) \leq \sum_{Q \in \mathscr{D}(P_0)} \beta_Q\Bigl( \prod_{i=1}^k \langle f_i \rangle_Q \Bigr) \mathbbm{1}_Q(x).
	\end{equation}
\end{lemma}
\begin{proof}
	We will prove by induction the following claim: if $P \in \mathscr{D}_{jm}(P_0)$ for some $j \geq 0$, then
	\begin{equation} \label{Complexity:Pointwise.Lemma.eq.induction}
		\mathcal{A}^{m;0}_{P,\alpha} \vec{f}(x) \leq \Delta_P + \sum_{Q \in \mathscr{D}(P)} \beta_Q \Bigl( \prod_{i=1}^k \langle f_i \rangle_Q \Bigr) \mathbbm{1}_Q(x).
	\end{equation}
	Note that, when $P = P_0$, this is exactly \eqref{Complexity:Pointwise.Lemma.eq}. Since $\alpha$ is finite,
        there is a smallest $j_0 \in \mathbb{N}$ such that $\alpha_Q = 0$ for all cubes 
        $Q \in \mathscr{D}_{\geq j_0m}(P_0)$\footnote{
        We use $\mathscr{D}_{\geq k}(P)$ to denote those cubes $Q$ in $\mathscr{D}(P)$ of generation at least $k$, so $|Q| \leq 2^{-dk}|P|$.}. Let $Q$ be any cube in
	$\mathscr{D}_{j_0m}(P_0)$, we obviously have
	\[
		\mathcal{A}^{m;0}_{Q,\alpha} \vec{f} \equiv 0 \quad \text{in } Q.
	\]
	Since $\Delta_Q \geq 0$, the claim \eqref{Complexity:Pointwise.Lemma.eq.induction} is trivial for $P \in \mathscr{D}_{j_0m}(P_0)$. Now, assume by induction that we have proved \eqref{Complexity:Pointwise.Lemma.eq.induction} for all cubes $P \in \mathscr{D}_{jm}(P_0)$ with $1\leq j_1 \leq j$
	and let $P$ be any cube in $\mathscr{D}_{(j_1-1)m}(P_0)$. By definition,
	\[
		\mathcal{A}^{m;0}_{P,\alpha} \vec{f}(x) = \sum_{Q \in \mathscr{D}_m(P)} \Bigl( \alpha_Q \Bigl( \prod_{i=1}^k \langle f_i \rangle_P \Bigr) \mathbbm{1}_Q(x)
			+ \mathcal{A}_{Q,\alpha}^{m;0}\vec{f}(x) \Bigr).
	\]
	Let $x \in Q \in \mathscr{D}_m(P)$, then by the induction hypothesis and the definition of $\Delta_Q$:
	\begin{align*}
		\mathcal{A}^{m;0}_{P, \alpha}\vec{f}(x) &\leq \alpha_Q \Bigl( \prod_{i=1}^k \langle f_i \rangle_P \Bigr) + \Delta_Q + 
			\sum_{R \in \mathscr{D}(Q)} \beta_R \Bigl( \prod_{i=1}^k \langle f_i \rangle_R \Bigr) \mathbbm{1}_R(x) \\
		&= \alpha_Q \Bigl( \prod_{i=1}^k \langle f_i \rangle_P \Bigr) + \Delta_P + 
			(\beta_P - \alpha_Q)\Bigl( \prod_{i=1}^k \langle f_i \rangle_P \Bigr) + 
				\sum_{R \in \mathscr{D}(Q)} \beta_R \Bigl( \prod_{i=1}^k \langle f_i \rangle_R \Bigr)\mathbbm{1}_R(x) \\
		&= \Delta_P + \beta_P \Bigl( \prod_{i=1}^k \langle f_i \rangle_P \Bigr) + \sum_{R \in \mathscr{D}(Q)} \beta_R \Bigl( \prod_{i=1}^k \langle f_i \rangle_R \Bigr) \mathbbm{1}_R(x) \\
		&= \Delta_P + \sum_{R \in \mathscr{D}(P)} \beta_R \Bigl( \prod_{i=1}^k \langle f_i \rangle_R \Bigr) \mathbbm{1}_R(x),
	\end{align*}
	which is what we wanted to show.
\end{proof}

\begin{lemma}
\label{lemmaunotres}
	The collection $\mathcal{S}$ is sparse.
\end{lemma}
\begin{proof}
	Let $P \in \mathcal{S}$, we have to show that the set
	\[
		F := \bigcup_{Q \subsetneq P, Q \in \mathcal{S}} Q
	\]
	satisfies $|F| \leq \frac{1}{2}|P|$. To this end, let $\mathcal{R}$ be the collection of maximal (strict) subcubes of $P$ which are in $\mathcal{S}$, Note that for all $R \in \mathcal{R}$ we have 
	$R \in \mathscr{D}_{N_Rm}(P)$ for some $N_R\geq1$. We thus have
	\[
		F = \bigsqcup_{R \in \mathcal{R}} R.
	\]
	
	By maximality, for all $R \in \mathcal{R}$ and dyadic cubes $Q$ with $R \subsetneq Q \subsetneq P$ we have $\beta_Q = 0$. For all $R\in \mathcal{R}$ and $1 \leq j \leq N_R$ we now claim that
	\begin{equation} \label{Complexity.Pointwise.Sparse.Delta.Estimate}
		\Delta_{R^{((N_R-j)m)}} \geq \beta_P\Bigl( \prod_{i=1}^k \langle f_i \rangle_P \Bigr) - \sum_{\nu=1}^j \alpha_{R^{((N_R-\nu)m)}} \Bigl( \prod_{i=1}^k \langle f_i \rangle_{R^{((N_R-\nu+1)m)}} \Bigr).
	\end{equation}
	Indeed, one can prove this by induction on $j$. If $j=1$ then by definition we have
	\begin{align*}
		\Delta_{R^{((N_R-1)m)}} &= \Delta_P + (\beta_P - \alpha_{R^{((N_R-1)m)}})\Bigl( \prod_{i=1}^k \langle f_i \rangle_P \Bigr) \\
		&\geq \beta_P \Bigl( \prod_{i=1}^k \langle f_i \rangle_P \Bigr) - \alpha_{R^{((N_R-1)m)}}\Bigl( \prod_{i=1}^k \langle f_i \rangle_P \Bigr),
	\end{align*}
	since $\Delta_P \geq 0$.
	
	To prove the induction step, observe that (by the induction hypothesis) for $j>1$
	\begin{align*}
		\Delta_{R^{((N_R-j)m)}} &= \Delta_{R^{((N_R-j+1)m)}} + (\beta_{R^{((N_R-j+1)m)}} - \alpha_{R^{((N_R-j)m)}}) \Bigl( \prod_{i=1}^k \langle f_i \rangle_{R^{((N_R-j+1)m)}} \Bigr) \\
		&= \Delta_{R^{((N_R-j+1)m)}} - \alpha_{R^{((N_R-j)m)}} \Bigl( \prod_{i=1}^k \langle f_i \rangle_{R^{((N_R-j+1)m)}} \Bigr) \\
		&\geq \beta_P\Bigl( \prod_{i=1}^k \langle f_i \rangle_P \Bigr) - \sum_{\nu=1}^{j} \alpha_{R^{((N_R-\nu)m)}} \Bigl( \prod_{i=1}^k \langle f_i \rangle_{R^{((N_R-\nu+1)m)}} \Bigr).
	\end{align*}
	
	From \eqref{Complexity.Pointwise.Sparse.Delta.Estimate} with $j = N_R$, we have (since the terms are nonnegative)
	\[
		\Delta_{R} \geq \beta_{P} \Bigl( \prod_{i=1}^k \langle f_i \rangle_P \Bigr) - \mathcal{A}^{m;0}_{P,\alpha}\vec{f}(x)
	\]
	for all $x \in R$. Since $\beta_{R} \neq 0$, we must have
	\[
		\Bigl( \prod_{i=1}^k \langle f_i \rangle_R \Bigr) \gamma_{R} - \Delta_{R} > 0,
	\]
	i.e.:
	\[
		\Bigl( \prod_{i=1}^k \langle f_i \rangle_R \Bigr) \gamma_{R} + \mathcal{A}^{m;0}_{P,\alpha}\vec{f}(x) > 2^{2(k+1)}C_W \Bigl( \prod_{i=1}^k \langle f_i \rangle_P \Bigr)
	\]
	for all $x \in R$. Let $\mathcal{G}_P\vec{f} = \sum_{R \in \mathcal{R}} \gamma_{R} \Bigl( \prod_{i=1}^k \langle f_i \rangle_R \Bigr) \mathbbm{1}_{R}$, then for all $x \in R$ we have
	\[
		\mathcal{G}_Pf(x) + \mathcal{A}^{m;0}_{P,\alpha}\vec{f}(x) > 2^{2(k+1)}C_W \Bigl( \prod_{i=1}^k \langle f_i \rangle_P \Bigr),
	\]
	hence
	\begin{align*}
		|F| &\leq \left|\left\{x \in P: \, \mathcal{G}_P \vec{f}(x) + \mathcal{A}^{m;0}_{P,\alpha} \vec{f}(x) > 2^{2(k+1)}C_W \Bigl( \prod_{i=1}^k \langle f_i \rangle_P \Bigr) \right\}\right| \\
		&\leq \frac{\|\mathcal{G}_P + \mathcal{A}_{P,\alpha}^{m;0}\|_{L^{1}(P) \times \dots \times L^1(P) \to L^{1/k,\infty}(P)}^{1/k}}{\Bigl(2^{2(k+1)}C_W  \Bigl( \prod_{i=1}^k \langle f_i \rangle_P \Bigr)\Bigr)^{1/k}} \Bigl( \prod_{i=1}^k \|f_i\|_{L^1(P)} \Bigr)^{1/k} \\
		&= \frac{\|\mathcal{G}_P + \mathcal{A}_{P,\alpha}^{m;0}\|_{L^{1}(P) \times \dots \times L^1(P) \to L^{1/k,\infty}(P)}^{1/k}}{(2^{2(k+1)}C_W)^{1/k}} |P|
	\end{align*}
	
	Let us compute the operator norm $\|\mathcal{G}_P\|_{L^1(P) \times \dots \times L^1(P) \to L^{1/k,\infty}(P)}$. Observe that, since $\gamma_Q \leq 1$ for all $Q$,
	the operator $\mathcal{G}$ is pointwise bounded
	by the multi-(sub)linear maximal operator introduced in \cite{Lerner2009a} (we have localized the operator in the obvious way):
	\[
		\mathcal{M}_P\vec{f}(x) = \sup \Bigl\{ \prod_{i=1}^k \frac{1}{|Q|} \int_Q |f_i(y)| \, dy:\, x \in Q \in \mathscr{D}(P) \Bigr\}.
	\]
	This operator is bounded from $L^1(P) \times \dots \times L^1(P) \to L^{1/k,\infty}(P)$, the proof can be found in \cite{Lerner2009a}. Therefore, we have
	\[
		\|\mathcal{M}_P \vec{f}\|_{L^{1/k,\infty}(P)}  \leq \prod_{i=1}^k \|f_i\|_{L^1(P)}.
	\]
	
	On the other hand we have
	\[
		\|\mathcal{A}_{P,\alpha}^{m;0} \vec{f}\|_{L^{1/k,\infty}(P)} \leq C_{W} \prod_{i=1}^k \|f_i\|_{L^1(P)}
	\]
	by Theorem W.1. Combining these estimates we get
	\[
		\|\mathcal{G}_P + \mathcal{A}_{P,\alpha}^{m;0}\|_{L^{1}(P) \times \dots \times L^1(P) \to L^{1/k,\infty}(P)} \leq 2^{k+1}(1+C_W) \leq 2^{k+2}C_W
	\]
	and the result follows.
\end{proof}

From lemmas \ref{lemmaunodos} and \ref{lemmaunotres} Proposition \ref{Complexity:Pointwise.MainTheorem} follows at once.
The proof shows that one can actually take $C_1=2^{2+k(6+d(2k-1))}$. We are now ready to finish the proof of Theorem A, which we state here in full generality:

\begin{theorem}
    Let $\alpha$ be a Carleson sequence and let $P_0$ be a dyadic cube. For every $k$-tuple of nonnegative integrable functions $f_1, \dots, f_k$ on $P$
    there exists a sparse collection $\mathcal{S}$ of cubes in $\mathscr{D}(P)$ such that
    \begin{equation*}
        \mathcal{A}^m_{P,\alpha} \vec{f}(x) \leq C_2 \sum_{Q \in \mathcal{S}} \Bigl( \prod_{i=1}^k \langle f_i \rangle_Q \Bigr) \mathbbm{1}_Q(x).
    \end{equation*}
\end{theorem}
\begin{proof}
	If $m=0$ we can just apply Proposition \ref{Complexity:Pointwise.MainTheorem} after noting that $\mathcal{A}^0_{P_0,\alpha}$ can be written as $\mathcal{A}^{1;0}_{P_0,\beta}$, where
	\[
		\beta_Q = \alpha_{Q^{(1)}}.
	\]
	One easily sees that $\|\alpha\|_{\car(P_0)} = \|\beta\|_{\car(P_0)}$. Hence, we may assume that $m \geq 1$. Recall the expression 
	\begin{align*}
		\mathcal{A}^m_{P_0,\alpha}\vec{f}(x) &= \sum_{n=0}^{m-1} \sum_{P \in \mathscr{D}_n(P_0)} \mathcal{A}^{m;0}_{P,\alpha} \vec{f}(x).
	\end{align*}
	from the beginning of the section. By Proposition \ref{Complexity:Pointwise.MainTheorem}, for each $0 \leq n \leq m-1$ and each $P \in \mathscr{D}_n(P_0)$ we can find a sparse
	collection of cubes $\mathcal{S}_{P}^n \subset \mathscr{D}(P)$ such that
	\[
		\mathcal{A}^{m;0}_{P,\alpha} \vec{f}(x) \leq C_1 \|\alpha\|_{\car(P_0)} \sum_{Q \in \mathcal{S}_P^n} \Bigl( \prod_{i=1}^k \langle f_i \rangle_Q \Bigr) \mathbbm{1}_Q(x).
	\]
	Observe that the collection $\mathcal{S}^n = \cup_{P \in \mathscr{D}_n(P_0)} \mathcal{S}_P^n$ is also sparse, so
	\begin{equation}\label{Complexity.Pointwise.eq}
		\mathcal{A}^m_{P_0,\alpha}\vec{f}(x) \leq C_1 \|\alpha\|_{\car(P_0)} \sum_{n=0}^{m-1} \sum_{Q \in \mathcal{S}^n} \Bigl( \prod_{i=1}^k \langle f_i \rangle_Q \Bigr) \mathbbm{1}_Q(x).
	\end{equation}
	For $0 \leq n \leq m-1$ define
	\[
		\mu_Q^n =
			\begin{cases}
				1 &\text{if } Q \in \mathcal{S}^n \\
				0 &\text{otherwise.}
			\end{cases}
	\]
	
	Since the collections $\mathcal{S}^n$ are sparse, the sequences $\mu^n$ are Carleson sequences with $\|\mu^n\|_{\car(P_0)} \leq 2$,
	therefore the sequence
	\[
		\mu_Q := \sum_{n=0}^{m-1} \mu_Q^n
	\]
	is also Carleson with $\|\mu\|_{\car(P_0)} \leq 2m$.
	
	With this we can continue the argument using estimate \eqref{Complexity.Pointwise.eq} and the case $m=0$:
	\begin{align*}
		\mathcal{A}^m_{P_0,\alpha}\vec{f}(x) &\leq C_1 \|\alpha\|_{\car(P_0)} \sum_{n=0}^{m-1} \sum_{Q \in \mathcal{S}^n} \Bigl( \prod_{i=1}^k \langle f_i \rangle_Q \Bigr) \mathbbm{1}_Q(x) \\
		&= C_1 \|\alpha\|_{\car(P_0)} \sum_{n=0}^{m-1} \sum_{Q \in \mathscr{D}(P_0)} \mu_Q^n \Bigl( \prod_{i=1}^k \langle f_i \rangle_Q \Bigr) \mathbbm{1}_Q(x) \\
		&= C_1 \|\alpha\|_{\car(P_0)} \sum_{Q \in \mathscr{D}(P_0)} \mu_Q \Bigl( \prod_{i=1}^k \langle f_i \rangle_Q \Bigr) \mathbbm{1}_Q(x) \\
		&= C_1 \|\alpha\|_{\car(P_0)} \mathcal{A}^0_{P_0,\mu} \vec{f}(x) \\
		&\leq C_1 \|\alpha\|_{\car(P_0)} C_1 2m \sum_{Q \in \mathcal{S}} \Bigl( \prod_{i=1}^k \langle f_i \rangle_Q \Bigr) \mathbbm{1}_Q(x),
	\end{align*}
	which yields the result with $C_2 = 2C_1^2$.
\end{proof}

We now detail how to use Theorem \ref{Complexity.TheoremA} to derive the multilinear version of corollaries \ref{Complexity.TheoremA.1} and \ref{Complexity.TheoremA.2}.
For us, a multilinear Calder\'on-Zygmund operator will be an operator $T$ satisfying
	\[
		T(f_1, \dots, f_k) = \int_{\mathbb{R}^{dk}} K(x, y_1, \dots, y_k) f_1(y_1) \cdot \dots \cdot f_k(y_k) dy_1 \, \dots \, dy_k
	\]
	for all $x \notin \cap_{i=1}^k \supp f_i$ for appropriate $f_i$. Also we will require that $T$ extends to a bounded operator from $L^{q_1} \times \dots L^{q_k}$ to $L^q$
	where
	\[
		\frac{1}{q} = \frac{1}{q_1} + \dots + \frac{1}{q_k},
	\]
	and that it satisfies the size estimate
	\[
		|K(y_0, \dots, y_k)| \leq \frac{A}{\Bigl(\sum_{i,j=0}^k |y_i-y_j|\Bigr)^{kd}}.
	\]

        $\omega$ will be the modulus of continuity of the kernel of the operator i.e. a positive nondecreasing continuous and doubling that satisfies
	\[
		|K(y_0, \dots, y_j, \dots, y_k)-K(y_0, \dots, y_j', \dots, y_k)| \leq C\omega\left(\frac{|y_j-y_j'|}{\sum_{i,j=0}^k |y_i-y_j|} \right) \frac{1}{\Bigl(\sum_{i,j=0}^k |y_i-y_j|\Bigr)^{kd}}
	\]
for all $0 \leq j \leq k$, whenever $|y_j-y_j'|\leq \frac{1}{2}\max_{0 \leq i \leq k}|y_j-y_i|$. We can now prove Corollary \ref{Complexity.TheoremA.1}:

\begin{proof}[Proof of Corollary \ref{Complexity.TheoremA.1}]
Fix a measurable $f$, and a cube $Q_0 \subset \mathbb{R}^d$. Our starting point is the formula
$$
|T\vec{f}(x)-m_{T\vec{f}}(Q_0)| \lesssim \sum_{Q \in \mathcal{S}} \sum_{m=0}^\infty \omega(2^{- m}) \prod_{i=1}^m \langle |f_i| \rangle_{2^m Q} \mathbbm{1}_Q(x),
$$
which holds for a sparse subcollection $\mathcal{S} \subset \mathscr{D}(Q_0)$ (see \cite{Damian2012} and \cite{Hytonen2013}, we are implicitly using a slight improvement of Lerner's formula which can be found in \cite{Hytonen2012c}, Theorem 2.3).
Here $m_f(Q)$ denotes the median of a measurable function $f$ over a cube $Q$ (see \cite{Lerner2013} for the precise definition), which satisfies
\begin{equation*}
|m_Q(f)| \lesssim \frac{\|f\|_{L^{1,\infty}(Q)}}{|Q|}.
\end{equation*} 
Hence we can just write 
\begin{equation}
\label{estimateonmedian}
|T\vec{f}(x)| \lesssim \sum_{m=0}^\infty \omega(2^{- m}) \prod_{i=1}^m \langle |f_i| \rangle_{2^m Q} \mathbbm{1}_Q(x),
\end{equation}
By an elaboration of the well-known one-third trick, it was proven in \cite{Hytonen2013} that there exist dyadic systems $\{\mathscr{D}^\rho\}_{\rho \in \{0,1/3,2/3\}^d}$ such that for every cube $Q$ in $\mathbb{R}^d$ and every $m\geq 1$, there exists $\rho \in \{0,1/3,2/3\}^d$ and $R_{Q,m} \in \mathscr{D}^\rho$ such that
$$
Q \subset R_{Q,m}, \; 2^mQ \subset Q^{(m)}, \; 3\ell(Q) < \ell(R_{Q,m}) \leq 6\ell(Q).
$$ 
Using this, we can further write \eqref{estimateonmedian} as 
\begin{align*}
	|T\vec{f}(x)| &\lesssim \sum_{\rho \in \{0,\frac{1}{3},\frac{2}{3}\}^d} 
		\sum_{m=0}^\infty \omega(2^{- m}) 
		\sum_{\begin{array}{c}\\[-6mm]\scriptstyle{Q \in \mathcal{S}}\\[-1.5mm] \scriptstyle{R_{Q,m} \in \mathscr{D}^\rho}\end{array}} 
		\Bigl( \prod_{i=1}^k \langle |f_i| \rangle_{R_{Q,m}^{(m)}} \Bigr) \mathbbm{1}_{R_Q}.
\end{align*}
Let $\mathcal{F}^\rho_m = \{R_{Q,m}: \, R_Q \in \mathscr{D}^\rho\}$. Then, we can estimate
\[
	|T\vec{f}(x)| \lesssim
		6^d  \sum_{\rho \in \{0,\frac{1}{3},\frac{2}{3}\}^d}  \sum_{m=0}^\infty \omega(2^{- m})
		 \sum_{R \in \mathcal{F}^\rho_m} \Bigl( \prod_{i=1}^k \langle |f_i| \rangle_{R^{(m)}} \Bigr) \mathbbm{1}_R,
\]
since at most $6^d$ cubes $Q$ in $\mathscr{D}$ are mapped to the same cube $R_{Q,m}$. Define the sequence
\begin{equation*}
    \alpha^\rho_Q = \begin{cases}
        1 &\text{if } Q \in \mathcal{F}^\rho_m \\
        0 &\text{otherwise}.
    \end{cases}
\end{equation*}
The collections $\mathcal{F}^\rho_m$ are $2^{-1}\cdot 6^{-d}$-sparse, and hence Carleson with constant
$2 \cdot 6^d$. 

Therefore, we can obtain the first assertion of Corollary \ref{Complexity.TheoremA.1} applying Theorem \ref{Complexity.TheoremA}:
\begin{align*}
	|T\vec{f}(x)| & \lesssim \sum_{\rho \in \{0,\frac{1}{3},\frac{2}{3}\}^d}  \sum_{m=0}^\infty \omega(2^{- m})
        \sum_{Q \in \mathscr{D}^\rho } \alpha_Q^\rho \Bigl( \prod_{i=1}^k \langle |f_i| \rangle_{Q^{(m)}} \Bigr) \mathbbm{1}_Q(x) \\
		& \lesssim \sum_{\rho \in \{0,\frac{1}{3},\frac{2}{3}\}^d}  \sum_{m=0}^\infty \omega(2^{- m}) (m+1)
		\sum_{Q \in \mathcal{S}_{m,\vec{f}}} \Bigl( \prod_{i=1}^k \langle |f_i| \rangle_Q \Bigr) \mathbbm{1}_Q \\
		& =  \sum_{\rho \in \{0,\frac{1}{3},\frac{2}{3}\}^d}  \sum_{m=0}^\infty \omega(2^{- m}) (m+1)
		\mathcal{A}_{\mathcal{S}_{m,\vec{f}}} \vec{f} (x),
\end{align*}
for sparse collections $\mathcal{S}_{m,\vec{f}}$ that may depend both on $m$ and $\vec{f}$. Now, reorganizing the sum above  we obtain
\begin{align*}
|T\vec{f}(x)| & \lesssim \sum_{\rho \in \{0,\frac{1}{3},\frac{2}{3}\}^d}  \sum_{\mathcal{S}_{m,\vec{f}} \subset \mathscr{D}^\rho} \omega(2^{- m}) (m+1)  
		\mathcal{A}_{\mathcal{S}_{m,\vec{f}}} \vec{f} (x) \\
& =: \sum_{\rho \in \{0,\frac{1}{3},\frac{2}{3}\}^d}  \mathcal{A}_{\rho}\vec{f}(x) .
\end{align*}
Now, by the logarithmic Dini condition, each of the operators $\mathcal{A}_{\rho}$ is bounded above by some absolute constant times a $0$-shift to which we can apply again Theorem \ref{Complexity.TheoremA}.
Therefore, we obtain
$$
|T\vec{f}(x)| \lesssim  \sum_{\rho \in \{0,\frac{1}{3},\frac{2}{3}\}^d} \mathcal{A}_{\mathcal{S}_\rho} \vec{f}(x),
$$
for some sparse families $\mathcal{S}_\rho \subset \mathscr{D}^\rho$ which depend on $\vec{f}$.
\end{proof}

We now introduce the notion of function quasi-norm. We say that $\| \cdot \|_{\mathbb{X}}$, defined on the set of measurable functions, is a function quasi-norm if:
	\begin{description}[style=multiline]
		\item[\namedlabel{Complexity.QuasiNorm.P1}{(P1)}] There exists a constant $C > 0$ such that
		\[
			\|f + g\|_{\mathbb{X}} \leq C \bigl( \|f\|_{\mathbb{X}} + \|g\|_{\mathbb{X}} \bigr),
		\]
		\item[\namedlabel{Complexity.QuasiNorm.P2}{(P2)}] $\| \lambda f \|_{\mathbb{X}} = |\lambda|\|f\|_{\mathbb{X}}$ for all $\lambda \in \mathbb{C}$.
		\item[\namedlabel{Complexity.QuasiNorm.P3}{(P3)}] If $|f(x)| \leq |g(x)|$ almost-everywhere then $\|f\|_{\mathbb{X}} \leq \|g\|_{\mathbb{X}}$.
		\item[\namedlabel{Complexity.QuasiNorm.P4}{(P4)}] $\|\liminf_{n \to \infty}f_n\|_{\mathbb{X}} \leq \liminf_{n \to \infty} \|f_n\|_{\mathbb{X}}$
	\end{description}
	
Now, taking into account properties (P1) and (P3), if we take quasi-norms in the second assertion of Corollary \ref{Complexity.TheoremA.1} and we get
	\begin{align*}
		\|T\vec{f}\|_X &\lesssim \sup_{\mathscr{D},\mathcal{S}}\Bigl\| \mathcal{A}_{\mathcal{S}}\vec{f} \Bigr\|_X.
	\end{align*}
This is exactly Corollary \ref{Complexity.TheoremA.2}.

\begin{remark}
Our method gives Corollary \ref{Complexity.TheoremA.2} only for kernels whose modulus of continuity satisfy the logarithmic Dini condition. However, if one could run the algorithm used in the proof of Proposition \ref{Complexity:Pointwise.MainTheorem} without slicing the operator $\mathcal{A}_{\alpha,P}^{m}$, then the factor $m$ in Theorem A could be removed, and therefore one would obtain Corollary \ref{Complexity.TheoremA.2} under the weaker assumption of the Dini condition
$$
\int_0^1 \omega(t)\; \frac{dt}{t} < \infty. 
$$
\end{remark}

\section{Applications}
\label{Complexity.Applications}
We are now ready to fully state and prove the applications of the pointwise bound as stated in the introduction. We begin with the multilinear sharp weighted estimates:

\subsection{Multilinear \texorpdfstring{$A_2$}{A2} theorem}

We need some more definitions first. These were introduced in \cite{Lerner2009a}.
\begin{definition}[$A_{\vec{P}}$ weights]
	Let $\vec{P} = (p_1, \dots, p_k)$ with $1 \leq p_1, \dots, p_k < \infty$ and $\frac{1}{p} = \frac{1}{p_1} + \dots + \frac{1}{p_k}$.
	Given $\vec{w} = (w_1, \dots, w_k)$, set
	\[
		v_{\vec{w}} = \prod_{i=1}^k w_{i}^{p/p_i}.
	\]
	We say that $\vec{w}$ satisfies the $k$-linear $A_{\vec{P}}$ condition if
	\[
		[\vec{w}]_{A_{\vec{P}}} := \sup_{Q} \Bigl( \frac{1}{|Q|}\int_Q v_{\vec{w}} \Bigr) \prod_{i=1}^k \Bigl( \frac{1}{|Q|}\int_Q w_i^{1-p_i'} \Bigr)^{p/p_i}.
	\]
	We call $[\vec{w}]_{A_{\vec{P}}}$ the $A_{\vec{P}}$ constant of $\vec{w}$. As usual, if $p_i = 1$ then we interpret
	$\frac{1}{|Q|}\int_Q w_i^{1-p_i'}$ to be $(\essinf_{Q} w_i)^{-1}$.
\end{definition}

The following theorem was proved in \cite{Li2014}:
\begin{theorem} \label{Complexity.LiMoenSun.Result}
	Suppose $1 < p_1, \dots, p_k < \infty$, $\frac{1}{p} = \frac{1}{p_1} + \dots + \frac{1}{p_k}$ and $\vec{w} \in A_{\vec{P}}$. Then
	\[
		\|\mathcal{A}_S \vec{f}\|_{L^{p}({v_{\vec{w}}})} \lesssim [w]_{A_{\vec{P}}}^{\max(1, \frac{p_1'}{p}, \dots, \frac{p_k'}{p})}
			\prod_{i=1}^k \|f_i\|_{L^p(w_i)},
	\]
	whenever $\mathcal{S}$ is sparse.
\end{theorem}

We can now use Corollary \ref{Complexity.TheoremA.2} to extend the above result to general $k$-linear Calder\'on-Zygmund operators:
\begin{theorem}
	Under the conditions of Theorem \ref{Complexity.LiMoenSun.Result}, for any $k$-linear Calder\'on-Zygmund operator $T$, we have
	\[
		\|T\vec{f}\|_{L^p(v_{\vec{w}})} \lesssim [\vec{w}]_{A_{\vec{P}}}^{\max(1, \frac{p_1'}{p}, \dots, \frac{p_k'}{p})}\prod_{i=1}^k \|f_i\|_{L^p(w_i)}.
	\]
\end{theorem}
\begin{proof}
	We just need to apply Corollary \ref{Complexity.TheoremA.2} with $\| \cdot \|_X := \| \cdot \|_{L^p(v_{\vec{w}})}$, which clearly
	is a function quasi-norm. The assumption of $\vec{f}$ being integrable is a qualitative one and can be trivially removed by the usual density arguments.
\end{proof}

\subsection{Sharp aperture weighted Littlewood-Paley theorem}

Here we follow Lerner \cite{Lerner2014}, the reader can find a nice introduction and some references there. We begin with some definitions:

Let $\psi \in L^1(\mathbb{R}^d$ with $\int_{\mathbb{R}^d} \psi(x) \, dx = 0$ satisfy
\begin{align}
	|\psi(x)| &\lesssim \frac{1}{(1+|x|)^{d+\epsilon}} \\
	\int_{\mathbb{R}^d} |\psi(x+h)-\psi(x)| \, dx &\lesssim |h|^\epsilon.
\end{align}

We will denote the upper half-space $\mathbb{R}^d \times \mathbb{R}$ by $\mathbb{R}^{d+1}_+$ and the $\alpha$-cone at $x$ by
\begin{equation*}
    \Gamma_\alpha(x) = \{(y,t) \in \mathbb{R}^{d+1}_+ :\, |y-x| \leq \alpha t\}.
\end{equation*}

Let $\psi_t$ be the dilation of $\psi$ which preserves the $L^1$ norm, i.e.: $\psi_t(x) = t^{-d} \psi(x/t)$, then we can define the square function
$S_{\alpha,\psi}f$ by
\[
	S_{\alpha,\psi}f(x) = \Bigl( \int_{\Gamma_\alpha(x)} |(f \ast \psi_t)(y)|^2 \, \frac{dy \, dt}{t^{d+1}} \Bigr)^{1/2}.
\]

We will also need a regularized version. Let $\Phi$ be a Schartz function such that
\[
	\mathbbm{1}_{B(0,1)}(x) \leq \Phi(x) \leq \mathbbm{1}_{B(0,2)}(x).
\]

We define the regularized square function $\widetilde{S}_{\alpha,\psi}$ by
\[
	\widetilde{S}_{\alpha,\psi}f(x) = \Bigl( \int_{\mathbb{R}^{d+1}_+} \Phi\Bigl( \frac{x-y}{t\alpha} \Bigr) |(f\ast \psi_t)(y)|^2 \, \frac{dy\, dt}{t^{d+1}} \Bigr)^{1/2}.
\]

The regularized version can be used instead of $S_{\alpha,\psi}$ in most cases since we have
\[
	S_{\alpha,\psi}f(x) \leq \widetilde{S}_{\alpha,\psi}f(x) \leq S_{\alpha,\psi}f(x).
\]

It was proved in \cite{Lerner2014} that
\[
	|(\widetilde{S}_{\alpha,\psi}f(x))^2 - (m_{Q_0}(\widetilde{S}_{\alpha,\psi}f)^2)| \lesssim \alpha^{2d} \sum_{m=0}^\infty 2^{-\delta m} \sum_{Q \in \mathcal{S}}
		\langle |f| \rangle_{2^m Q}^2 \mathbbm{1}_Q(x)
\]

By the same Theorem \ref{Complexity.TheoremA} in its bilinear formulation (with $f_1=f_2=f$), the last expression can be bounded, up to a constant, by an expression of the form
\begin{equation*}
    \alpha^{2d}  \sum_{\rho \in \{0,\frac{1}{3},\frac{2}{3}\}^d} \sum_{m=0}^\infty 2^{-\delta m} (m+1) \sum_{Q \in \mathcal{S}^{\rho,m}} \langle |f| \rangle_Q^2 \mathbbm{1}_Q(x).
\end{equation*}

As in \cite{Lerner2014}, we know (a priori) that $m_{Q_0}(\widetilde{S}_{\alpha,\psi}f) \to 0$ as $|Q| \to \infty$ so by the triangle inequality and Fatou's lemma we can ignore that term
(or by arguing as we did in the previous section). Finally, arguing as in the proof of corollaries \ref{Complexity.TheoremA.1} and \ref{Complexity.TheoremA.2}, we arrive at
\begin{equation*}
    \|\widetilde{S}_{\alpha,\psi}f\|_{L^{p,\infty}(w)} \lesssim \alpha^d \sup_{\mathscr{D}, \mathcal{S}} \|\mathcal{A}^0_{\mathcal{S}}(f,f)^{1/2}\|_{L^{p,\infty}(w)},
\end{equation*}
where the supremum is taken over all dyadic grids $\mathscr{D}$ and all sparse collections $\mathcal{S} \subset \mathscr{D}$. To finish the argument we recall the following result, which was shown in \cite{Lacey2012}:
\begin{equation}
    \label{eq:LaceyScurry}
    \|\mathcal{A}^0_{\mathcal{S}}(f,f)^{1/2}\|_{L^{p,\infty}(w)} \lesssim [w]_{A_p}^{\max(\frac{1}{2}, \frac{1}{p})} \Phi_p([w]_{A_p}) \|f\|_{L^p(w)}
\end{equation}
for $1 < p < 3$, where
\begin{equation*}
    \Phi_p(t) = 
    \begin{cases}
        1 &\text{if } 1<p<2 \\
        1+\log t &\text{if } 2 \leq p < 3.
    \end{cases}
\end{equation*}

We are thus able to extend Lerner's estimate to $1 < p \leq 2$, obtaining
\begin{equation*}
    \|S_{\alpha,\psi}f\|_{L^{p,\infty}(w)} \lesssim \alpha^d [w]_{A_p}^{1/p} \|f\|_{L^p(w)} \quad \text{for }1 < p < 2
\end{equation*}
and
\begin{equation*}
    \|S_{\alpha,\psi}f\|_{L^{2,\infty}(w)} \lesssim \alpha^d [w]_{A_2}^{1/2} (1+\log[w]_{A_2}) \|f\|_{L^2(w)}.
\end{equation*}

\appendix

\renewcommand\thesection{W}

\section{The weak-type estimate for multilinear \texorpdfstring{$m$}{m}-shifts}
\label{Complexity.Appendix}
Here we prove the weak-type estimate for $k$-linear $m$-shifts needed in section \ref{Complexity.Pointwise}.
Notice that the only important point of the calculations below is the independence of the constants from the parameter $m$.
The proof is more or less standard by now, but the authors have not been able to find a proof of this result elsewhere. Therefore we include it for completeness.
\begin{theorem}\label{Complexity.LebesgueWeakType.Theorem}
\begin{equation}
	\sup_{\lambda > 0} \lambda |\{x\in P_0: \mathcal{A}_{P_0,\alpha}^m \vec{f}(x) > \lambda \}|^{k} \leq C_{W} \|\alpha\|_{\car(P_0)} \prod_{i=1}^k \|f_i\|_{L^1(P_0)},
\end{equation}
where $C_W > 0$ only depends on $k$ and $d$, and in particular is independent of $m$.
\end{theorem}

We will essentially follow Grafakos-Torres \cite{Grafakos2002} and \cite{Hytonen2012a}.
We first prove an $L^{2}$ bound and then apply a Calder\'on-Zygmund decomposition.
For the $L^{2}$ bound we will use a multilinear Carleson embedding theorem by W. Chen and W. Dami\'an \cite{Chen2013}, however, we only need the unweighted result:

\begin{equation}\label{Complexity.MultilinearCET} 
    \Bigl( \sum_{Q \in \mathscr{D}(P_0)} \alpha_Q \Bigl( \prod_{i=1}^k \langle f_i \rangle_Q  \Bigr)^p  \Bigr) \leq \|\alpha\|_{\text{Car}(P_0)} \prod_{i=1}^k p_i' \|f_i\|_{L^{p_i}(P_0)}
\end{equation}
whenever
\begin{equation*}
    \frac{1}{p} = \frac{1}{p_1} + \dots + \frac{1}{p_k}.
\end{equation*}

Now we can prove
\begin{proposition}\label{Complexity:LebesgueL2Bound.Theorem}
	\[
		\|\mathcal{A}_{P_0,\alpha}^m \vec{f}\|_{L^2(P_0)} \leq 4\|\alpha\|_{\car(P_0)}\prod_{i=1}^k \|f_i\|_{L^{2k}(P_0)}
	\]
\end{proposition}

\begin{proof}
	We begin by using duality and homogeneity to reduce to showing
	\[
		\int_{P_0} g(x) \mathcal{A}_{P_0,\alpha}^m \vec{f}(x) \, dx \leq 4
	\]
	assuming that $\|f_i\|_{L^{2k}(P_0)} = \|g\|_{L^2(P_0)} = \|\alpha\|_{\car(P_0)} = 1$ and $g\geq 0$.
	
	By definition and Cauchy-Schwarz, this is equivalent to
	\[
		\Bigl( \sum_{Q \in \mathscr{D}_{\geq m}(P_0)} \alpha_Q \Bigl( \prod_{i=1}^k \langle f_i \rangle_{Q^{(m)}} \Bigr)^2 |Q| \Bigr)^{1/2} 
			\Bigl( \sum_{Q \in \mathscr{D}_{\geq m}(P_0)} \alpha_Q \langle g \rangle_Q^2 |Q| \Bigr)^{1/2}.
	\]
	The second term can be estimated, using \eqref{Complexity.MultilinearCET} in the linear case, by
	\[
		\Bigl( \sum_{Q \in \mathscr{D}_{\geq m}(P_0)} \alpha_Q \langle g \rangle_Q^2 |Q| \Bigr)^{1/2} \leq 2.
	\]
	
	For the first term observe that the sequence $\beta_Q$ defined by
	\[
		\beta_Q = \frac{1}{2^{dm}} \sum_{R \in \mathscr{D}_m(Q)} \alpha_R
	\]
	is a Carleson sequence adapted to $P_0$ of the same constant. Indeed:
	\begin{align*}
		\frac{1}{|Q|} \sum_{R \in \mathscr{D}(Q)} \beta_R|R| &= \frac{1}{|Q|}\sum_{R \in \mathscr{D}(Q)} |R| \frac{1}{2^{dm}} \sum_{T \in \mathscr{D}_m(R)} \alpha_T \\
		&= \frac{1}{|Q|}\sum_{R \in \mathscr{D}(Q)}\sum_{T \in \mathscr{D}_m(R)} \alpha_T |T| \\
		&= \frac{1}{|Q|} \sum_{R \in \mathscr{D}_{\geq m}(Q)} \alpha_R|R| \\
		&\leq \|\alpha\|_{\text{Car}(I)} \\
		&=1.
	\end{align*}
	
	Therefore, we can write the first term as
	\[
		\Bigl( \sum_{Q \in \mathscr{D}(P_0)} \beta_Q \Bigl( \sum_{i=1}^k \langle f_i \rangle_Q  \Bigr)^2 |Q| \Bigr)^{1/2},
	\]
	which can also be estimated by \eqref{Complexity.MultilinearCET} as follows:
	\[
		\Bigl( \sum_{Q \in \mathscr{D}(P_0)} \beta_Q \Bigl( \sum_{i=1}^k \langle f_i \rangle_Q  \Bigr)^2 |Q| \Bigr)^{1/2} \leq \Bigl( \frac{2k}{2k-1} \Bigr)^k \leq 2.
	\]
	Combining both terms we arrive at
	\[
		\int_{P_0} g(x) \mathcal{A}_{P_0,\alpha}^m \vec{f}(x) \, dx \leq 4
	\]
	which is what we wanted.
\end{proof}

Now we can prove Theorem \ref{Complexity.LebesgueWeakType.Theorem}.
\begin{proof}
	By homogeneity we can assume $\|\alpha\|_{\car(P_0)} = \|f_i\|_{L^1(P_0)} = 1$. We now follow the classical scheme which uses the $L^2$ bound and a standard
	Calder\'on-Zygmund decomposition, see for example Grafakos-Torres \cite{Grafakos2002}. However, we need to be careful with the dependence on $m$, so we will adapt the
        proof in \cite{Hytonen2012a} to our operators.
	
	Assume without loss of generality that $f_i \geq 0$. Define
	\[
		\Omega_i = \{x\in P_0: \mathcal{M}^d f_i(x) > \lambda^{1/k}\}.
	\]
	If $\langle f_i \rangle_{P_0} > \lambda^{1/k}$ then by the homogeneity assumption
	\[
		|P_0| < \lambda^{-1/k}
	\]
	and the estimate follows. Therefore, we can assume $\langle f_i \rangle_{P_0} \leq \lambda^{1/k}$ for all $1 \leq i \leq k$ and hence we 
	can write $\Omega_i$ as a union the cubes in a collection $\mathcal{R}_i$ consisting of pairwise disjoint dyadic (strict) subcubes of $P_0$ with the property
	\[
		\langle f_i \rangle_R > \lambda^{1/k} \quad \text{and} \quad \langle f_i \rangle_{R^{(1)}} \leq \lambda^{1/k}.
	\]
	
	For each $1 \leq i \leq k$ let $b_i = \sum_{R \in \mathcal{R}_i} b_i^R$, where
	\[
		b_i^R(x) := \bigl( f_i(x) - \langle f_i \rangle_R \bigr) \mathbbm{1}_R(x).
	\]
	We now let $g_i = f_i-b_i$.
	
	Observe that we have
	\[
		|g_i(x)| \leq 2^d \lambda^{1/k},
	\]
	as well as
	\[
		|\Omega_i| = \sum_{R \in \mathcal{R}_i} |R| \leq \lambda^{-1/k}.
	\]
	
	Define $\Omega = \cup_{i=1}^k \Omega_i$, then we have
	\begin{align}
		|\{x \in P_0: \, \mathcal{A}_{P_0,\alpha}^m \vec{f}(x) > \lambda\}| &\leq |\Omega| + 
			|\{x\in P_0\setminus \Omega: \, \mathcal{A}_{P_0,\alpha}^m \vec{f}(x) > \lambda\}| \notag \\
		&\leq k \lambda^{-1/k} + |\{x\in P_0\setminus \Omega: \, \mathcal{A}_{P_0,\alpha}^m \vec{f}(x) > \lambda\}|. \label{Complexity.LebesgueWeakType.eq1}
	\end{align}
	
	To estimate the second term observe that
	\begin{align*}
		\mathcal{A}_{P_0,\alpha}^m \vec{f}(x) &= \mathcal{A}_{P_0,\alpha}^m (\vec{g} + \vec{b})(x) \\
		&= \mathcal{A}_{P_0,\alpha}^m \vec{g}(x) + \sum_{j=1}^{2^k-1} \mathcal{A}_{P_0,\alpha}^m(h_1^j, \dots, h_k^j)(x),
	\end{align*}
	where the functions $h_i^j$ are either $g_i$ or $b_i$ and, furthermore, for each $1 \leq j \leq 2^k-1$ there is at least one $1 \leq i \leq k$
	such that $h_i^j = b_i$.
	
	Fix $j$ and let $i_j$ be such that $h_{i_j}^j = b_{i_j}$, then
	\begin{align*}
		\mathcal{A}_{P_0}^m (h_1^j, h_2^j, \dots, h_{i_j}^j, \dots, h_k^j)(x) &= \sum_{Q \in \mathscr{D}_{\geq m}(P_0)}
			\alpha_Q \Bigl( \prod_{i=1}^k \langle h_i^j \rangle_{Q^{(m)}} \Bigr) \mathbbm{1}_Q(x) \\
		&= \sum_{Q \in \mathscr{D}_{\geq m}(P_0)}
			\alpha_Q \langle b_{i_j} \rangle_{Q^{(m)}}\Bigl( \prod_{1 \leq i \leq k, \, i \neq i_j} \langle h_i^j \rangle_{Q^{(m)}} \Bigr) \mathbbm{1}_Q(x) \\
		&= \sum_{R \in \mathcal{R}_{i_j}} \sum_{Q \in \mathscr{D}_{\geq m}(P_0)}
			\alpha_Q \langle b_{i_j}^R \rangle_{Q^{(m)}}\Bigl( \prod_{1 \leq i \leq k, \, i \neq i_j} \langle h_i^j \rangle_{Q^{(m)}} \Bigr) \mathbbm{1}_Q(x) \\
		&= \sum_{R \in \mathcal{R}_{i_j}} \sum_{Q \in \mathscr{D}_{> m}(R)}
			\alpha_Q \langle b_{i_j}^R \rangle_{Q^{(m)}}\Bigl( \prod_{1 \leq i \leq k, \, i \neq i_j} \langle h_i^j \rangle_{Q^{(m)}} \Bigr) \mathbbm{1}_Q(x).
	\end{align*}
	So we deduce that $\mathcal{A}_{P_0,\alpha}^m (h_1^j, \dots, h_k^j)(x) = 0$ for all $x \notin \Omega_{i_j}$. With this fact we can see that the
	second term in \eqref{Complexity.LebesgueWeakType.eq1} is actually identical to
	\[
		|\{x\in P_0\setminus \Omega: \, \mathcal{A}_{P_0,\alpha}^m \vec{g}(x) > \lambda\}|.
	\]
	
	Now we can use the $L^2$ bound as follows:
	\begin{align*}
		|\{x\in P_0\setminus \Omega: \, \mathcal{A}_{P_0,\alpha}^m \vec{g}(x) > \lambda\}| &\leq \frac{1}{\lambda^2} \|\mathcal{A}_{P_0,\alpha}^m \vec{g}\|_{L^2(P_0)}^2 \\
		&\leq \frac{16}{\lambda^2} \prod_{i=1}^k \|g_i\|_{L^{2k}(P_0)}^{2} \\
		&\leq \frac{16}{\lambda^2} \prod_{i=1}^k \bigl( 2^d \lambda^{1/k} \bigr)^{\frac{2k-1}{k}} \|g_i\|_{L^1(P_0)}^{1/k} \\
		&= \frac{16}{\lambda^2} 2^{d(2k-1)} \lambda^{2-1/k} \\
		&= 2^{4+ d(2k-1)}\lambda^{-1/k}.
	\end{align*}
	
	Putting both estimates together we arrive at
	\[
		|\{x \in P_0: \, \mathcal{A}_{P_0,\alpha}^m \vec{f}(x) > \lambda\}| \leq 2^{5+ d(2k-1)}\lambda^{-1/k}
	\]
	which yields the result with $C_{W} = 2^{k(5+ d(2k-1))}$.
\end{proof}

\bibliography{bibliography}{}
\bibliographystyle{abbrv}

\end{document}